\documentclass[a4paper,reqno]{amsart}
\usepackage{biblatex}
\usepackage{amsfonts}
\usepackage{amsthm}
\usepackage{amsmath}
\usepackage{datetime}
\usepackage{comment}
\usepackage{enumitem}
\usepackage[dvipsnames]{xcolor}
\usepackage{tikz}
\usetikzlibrary{plotmarks}
\usetikzlibrary{decorations.pathreplacing}
\numberwithin{equation}{section}

\newtheorem{theorem}{Theorem}[section]
\newtheorem{corollary}[theorem]{Corollary}
\newtheorem{lemma}[theorem]{Lemma}
\newtheorem{remark}[theorem]{Remark}
\newtheorem{definition}[theorem]{Definition}

\numberwithin{figure}{section}

\title[Asymptotics of $q$-orthogonal polynomials]{Asymptotic behaviours of $q$-orthogonal polynomials from a $q$-Riemann Hilbert Problem}

\author{Nalini Joshi}
\email{nalini.joshi@sydney.edu.au}
\address{School of Mathematics and Statistics F07, University of Sydney, Sydney NSW 2006, Australia, ORCID ID: 0000-0001-7504-4444}

\author{Tomas Lasic Latimer}
\email{tlas5434@uni.sydney.edu.au}
\address{School of Mathematics and Statistics F07, University of Sydney, Sydney NSW 2006, Australia, ORCID ID: 0000-0001-6859-7788}

\date{}
\begin{document}
\begin{abstract}
     We describe a Riemann-Hilbert problem for a family of $q$-orthogonal polynomials, $\{ P_n(x) \}_{n=0}^\infty$, and use it to deduce their asymptotic behaviours in the limit as the degree, $n$, approaches infinity. We find that the $q$-orthogonal polynomials studied in this paper share certain universal behaviours in the limit $n\to\infty$. In particular, we observe that the asymptotic behaviour near the location of their smallest zeros, $x \sim q^{n/2}$, and norm, $\|P_n\|_2$, are independent of the weight function as $n\to\infty$.
\end{abstract}


\maketitle
\tableofcontents

\section{Introduction}
The theory of orthogonal polynomials is a source of major developments in modern mathematical physics. But the spectacular outcomes of the classical theory of orthogonal polynomials with continuous measure are not yet matched by $q$-orthogonal polynomials, which are orthogonal with respect to a discrete measure supported on a lattice $\{\pm q^k\}_{k\in \mathbb{N}^+}$, for some $0<q<1$. In this paper, we focus on a class of such polynomials and deduce their asymptotic behaviours as their degree grows by expressing them in terms of a Riemann-Hilbert Problem (RHP). 

We denote monic $q$-orthogonal polynomials by $\bigl\{ P_{n}(x)\bigr\}_{n=0}^\infty$. They satisfy the orthogonality relation
\begin{equation}\label{0th orthogonality}
    \int^1_{-1} P_{n}(x)P_{m}(x) w(x) d_qx = \gamma_{n}\delta_{n,m} ,
\end{equation} 
where $d_qx$ refers to the (discrete) Jackson integral (see Equation \eqref{Jackson}). A fundamental consequence of Equation \eqref{0th orthogonality} is the 3-term recurrence relation
\begin{equation}\label{recurrence}
    xP_{n}(x) = P_{n+1}(x)+b_{n}P_{n}(x)+a_{n}P_{n-1}(x) \,,
\end{equation}
where the recurrence coefficients are given by 
\begin{eqnarray*}
a_{n} = \dfrac{\gamma_{n}}{\gamma_{n-1}} ,\ 
b_{n} = \dfrac{\int xP_{n}(x)^{2} d\mu(x)}{\gamma_{n}}.
\end{eqnarray*}

A natural question, which motivated many studies through the past century \cite{szeg1939orthogonal,Deift1999strong,ismail2005asymptotics}, is to ask what the behaviours of $P_n, \gamma_n, a_n$ and $b_n$ are as $n\rightarrow \infty.$ Classical results focused on polynomials  which satisfy the orthogonality relation
\begin{equation}\nonumber
    \int P_{n}(x)P_{m}(x) w(x) d\mu(x) = \gamma_{n}\delta_{n,m} ,
\end{equation}
where $d\mu(x)$ is a continuous measure on the real line and $w(x)$ is a weight function whose rate of change satisfies certain conditions. Measures on the unit circle in the complex plane were also a focus of interest \cite{szeg1939orthogonal,simon2004orthogonal}. Extensions to a wider class of weight functions, leading to so-called \textit{semi-classical} orthogonal polynomials \cite{shohat1939,freud1976coefficients}, have attracted more attention in recent times, due to their appearance in Random Matrix Theory \cite{mehtaRMT,D99} and relationship to the Painlev\'e equations \cite{fokas1991discrete}.  

More recently, discrete orthogonal polynomials have been of growing interest. For those on a multiplicative $q$-lattice, particular attention has been paid to cases such as little $q$-Jacobi and discrete $q$-Hermite \rm{I} \cite[Chapter 18.27]{NIST:DLMF} polynomials. However, very little appears to be known about the asymptotic behaviour of $q$-orthogonal polynomials outside these specific cases. More recent discoveries of $q$-orthogonal polynomials related to multiplicative discrete Painlev\'e equations \cites{filipuk2018discrete,Boelen} have reignited a need for further mathematical tools to answer questions about their asymptotic behaviours. 

The aim of this paper is to consider such questions for a general class of $q$-orthogonal polynomials which includes, but is not limited to, a large subset of the $q$-Hahn class \cite[Chapter 18.27]{NIST:DLMF}. Our main results are Theorems \ref{main result 1} and \ref{main result 2}, where, under certain mild assumptions on the weight $w(x)$ in Equation \eqref{0th orthogonality}, we deduce the asymptotic behaviour of $P_n$, $\gamma_n$ and $a_n$ in the limit $n\to \infty$ and show that the error term is of size $O(q^n)$. 

\subsection{Background}
In the past two decades, $q$-orthogonal polynomials have appeared in many areas of applied mathematics and physics \cites{atakishiyev2008discrete,jafarov2010quantum,sasaki2009exactly, knizel2016moduli}, particularly in quantum physics. However, little is known about the behaviour of $q$-orthogonal polynomials. Earlier work in the field focused on specific examples. In 2003, Postelmans and Van Assche introduced two kinds of multiple little $q$-Jacobi polynomials and described some asymptotic properties \cite{postelmans2005multiple}. In 2005, Ismail described the asymptotic behaviour of discrete $q$-Hermite \rm{II} polynomials using $q$-Airy functions. He then extended these results to $q$-orthogonal polynomials satisfying a certain $q$-difference equation \cite{ismail2005asymptotics}. In 2013, Driver and Jordaan studied the asymptotic behaviour of extreme zeros of $q$-orthogonal polynomials \cite{driver2013inequalities}. In 2017, Chen and Filipuk studied generalised $q$-Laguerre polynomials and determined resulting asymptotics for their recurrence coefficients \cite{chen2017nonlinear}.

Recently, there has been improved understanding of the asymptotics of larger families of $q$-orthogonal polynomials. In 2020, Van Assche \textit{et al.} \cite{van2020zero} showed that the leading order of $\gamma_n$ is $q^{n^{2}}$ for $q$-orthogonal polynomials with weights satisfying 
\[ \lim_{n\to \infty}\frac{1}{n^2}\text{log}(w(x^n)) = 0,\]
where $x\in (0,1)$. They also describe the location of $\{x_i^{1/n}\}_{i=1}^n$ (where $\{x_i\}_{i=1}^n$ are the $n$ zeros of $P_n$) in the limit $n\rightarrow \infty$. However, it remained an open question to obtain a more precise asymptotic description of large classes of $q$-orthogonal polynomials. 

RHPs have been extensively used to study the asymptotics of orthogonal polynomials \cites{kuijlaars2003riemann,etna_vol25_pp369-392} since the asymptotic behaviour of semi-classical Freudian polynomials was derived by Deift \textit{et al.} using a model RHP \cite{Deift1999strong}. Their work was based on earlier advancements by Deift and Zhou on the steepest descent method for oscillatory RHPs \cite{deift1993steepest}. Expanding on the approach of Deift \textit{et al.}, Baik \textit{et al.} \cite{baik2007discrete} deduced the asymptotics of orthogonal polynomials on a discrete lattice using what they call an interpolation problem, which can be seen as the discrete analogue to a RHP. Although this work yielded interesting results for general discrete weights, we find that it misses some key details of the behaviour of $q$-orthogonal polynomials. In particular, the results do not accurately describe the behaviour of $a_n$, $\gamma_n$ and $P_n$ as $n\rightarrow \infty$. We note that in terms of Baik \textit{et al.}'s notation, $0$ is an accumulation point of the $q$-lattice.

Extending on our earlier theory \cite{qRHP}, in this paper we will use a RHP to obtain detailed asymptotic results for a large class of $q$-orthogonal polynomials. We will observe an interesting intersection with $q$-RHP theory and provide explicit examples highlighting aspects of $q$-RHP theory discussed in the literature \cite{ramis2013local,sauloy2003galois,Adams}. In particular, we will solve the model $q$-RHP by deducing an equivalent connection matrix between two solutions of a $q$-difference equation represented by a power series about $0$ and $\infty$. 

The asymptotic results obtained in this paper also pertain to multiplicative discrete Painlev\'e equations. It has been shown that the recurrence coefficients of $q$-orthogonal polynomials can satisfy multiplicative-type discrete Painlev\'e equations \cite{Boelen}, for example Equation \eqref{multi pain}
 \begin{equation}\label{multi pain}
     a_n(a_{n+1} + q^{1-n}a_n+q^2a_{n-1} + q^{3-2n}a_{n+1}a_{n}a_{n-1}) = q^{n-1}(1-q^n),
 \end{equation}
 where the non-autonomous term in the equation is iterated on multiplicative lattices. (For the terminology distinguishing types of discrete Painlev\'e equations, we refer to Sakai \cite{s:01}.) 

Very little is known about the asymptotic behaviour of the solution to this equation. The results in this paper provide detailed asymptotics for the real positive solution to Equation \eqref{multi pain}.

\subsection{Notation and previous results in the literature}
For completeness, we recall some well known definitions and notations from the calculus of $q$-differences. These definitions can be found in \cite{Ernst2012}. Throughout the paper we will assume $q \in \mathbb{R}$ and $0<q<1$.
\begin{definition}
We define the Pochhammer symbol $(z;q)_{\infty}$, and Jackson integrals as follows.
\begin{enumerate}
\item The Pochhammer symbol $(z;q)_{\infty}$ is defined as
\begin{equation*}
    (z;q)_{\infty} = \prod_{j=0}^{\infty}(1-zq^{j}) \,.
\end{equation*}
Furthermore, we define $(z_1,z_2;q)_{\infty}$ as 
\begin{equation*}
    (z_1,z_2;q)_{\infty} = \prod_{j=0}^{\infty}(1-z_1q^{j})(1-z_2q^{j}) \,.
\end{equation*}
\item The unnormalised Jackson integral of $f(z)$ from -1 to 1 is defined as
\begin{equation}\label{Jackson}
    \int_{-1}^{1}f(z)d_{q}z = \sum_{k=0}^{\infty} (f(q^{k})+f(-q^{k}))q^{k} \,.
\end{equation}
\end{enumerate}
\end{definition}

We remark on an equivalence between two types of $q$-orthogonal polynomials seen in the literature.

\begin{remark}\label{201}
In general $q$-orthogonal polynomials can be orthogonal with respect to a weight supported on the Jackson integral from $[-1,1]$ or from $(0,1]$ \cite[Chapter 18.27]{NIST:DLMF}, where the latter is given by 
\begin{equation}\nonumber
    \int_{0}^{1}f(z)d_{q}z = \sum_{k=0}^{\infty} f(q^{k})q^{k} \,.
\end{equation}
Let $\{ P_n \}_{n=0}^\infty$ be the normalised polynomials orthogonal with respect to the one-sided Jackson integral
\begin{equation*}
    \sum_{k=0}^{\infty} P_n(q^k)P_m(q^k)w(q^k)q^k = \delta_{n,m}.
\end{equation*}
Let $\rho=q^{1/2}$, hence
\begin{equation}\label{one way orthogonality}
    \sum_{k=0}^{\infty} P_n(\rho^{2k})P_m(\rho^{2k})w(\rho^{2k})\rho^{2k} = \delta_{n,m}.
\end{equation}
Define $\omega(z) = w(z^2)|z|$, it follows that $\omega(z)$ is an even function. This implies that the corresponding set of orthogonal polynomials $\{Q_n(z)\}_{n=0}^\infty$ are even/odd for even/odd $n$ \cite{ismail2005classical}. Thus, for positive integers $l$, $p$ the orthogonality condition for $\{Q_n(z)\}_{n=0}^\infty$ is given by 
\begin{equation}\label{two way orthogonality}
    2\sum_{k=0}^{\infty} Q_{2l}(\rho^{2k})Q_{2p}(\rho^{2k})\omega(\rho^k)\rho^{k} = \delta_{l,p},
\end{equation}
which is equivalent to Equation \eqref{one way orthogonality} $($up to scaling of the normalisation factor by $\sqrt{2})$. Hence, we proved that the class of $q$-orthogonal polynomials with one-sided Jackson integrals are contained in the class of $q$-orthogonal polynomials with two-sided Jackson integrals. It follows that it is sufficient to study $q$-orthogonal polynomials with two-sided Jackson integrals.
\end{remark}

We recall the definition of an {\em appropriate} Jordan curve and {\em admissible} weight function given in \cite[Definition 1.2]{qRHP} (with slight modification).
\begin{definition}\label{admissable}
A positively oriented Jordan curve $\Gamma$ in $\mathbb C$ with interior $\mathcal D_-\subset\mathbb C$ and exterior $\mathcal D_+\subset\mathbb C$ is called {\em appropriate} if 
\[
\pm q^k\in \begin{cases}
&\mathcal D_- \quad {\rm if}\, k\ge 0,\\
&\mathcal D_+ \quad {\rm if}\, k< 0.
\end{cases}
\]
A weight function, $w(z)$, is called {\em admissible} if there exists an appropriate Jordan curve, $\Gamma$, such that $w(z)$ is bounded on $\Gamma$, $w(z)$ is analytic in $\mathcal D_-$, and, there exists constants $N_c$ and $c$ such that for $n>N_c$
\[ |1-w(\pm q^{n/2})| < cq^{n}.\]
Furthermore, we require that
\[ w(\pm q^k) \neq 0,\; {\rm for}\,k\in\mathbb{N}_0.\]
\end{definition}

We define the function 
\begin{equation}\label{h new def}
h^\alpha(z) = \sum_{k=-\infty}^{\infty} \frac{2zq^{k(1+\alpha)}}{z^{2}-q^{2k}}  = \sum_{k=-\infty}^{\infty} \left( \frac{q^{k(1+\alpha)}}{z-q^{k}} + \frac{q^{k(1+\alpha)}}{z+q^{k}} \right),
\end{equation}
which satisfies the $q$-difference equation 
\begin{equation}\label{h alpha diff}
    h^\alpha(qz) = q^{\alpha}h^\alpha(z).
\end{equation}
Note that $h^0(z)$ is equivalent to the definition of $h(z)$ given in \cite[Definition 1.3]{qRHP}. Consequently, we define the $q$-RHP:
\begin{definition}[$q$-RHP]\label{qRHP def}
Let $\Gamma$ be an appropriate curve (see Definition \ref{admissable}) with interior $\mathcal D_-$ and exterior $\mathcal D_+$, and $w(z)$ be a corresponding admissible weight. A $2\times 2$ complex matrix function $Y_n(z)$, $z\in\mathbb C$, is a solution to the $q$-RHP if it satisfies the following conditions:
    \begin{enumerate}[label={{\rm (\roman *)}}]
    \item $Y_n(z)$ is analytic on $\mathbb{C}\setminus \Gamma$.
\item $Y_n(z)$ has continuous boundary values $Y_n^-(s)$ and $Y_n^+(s)$ as $z$ approaches $s\in\Gamma$ from $\mathcal D_-$ and $\mathcal D_+$ respectively, where
\begin{subequations}
\begin{gather} \label{b}
 Y_n^+(s) =
  Y_n^-(s)
  \begin{pmatrix}
   1 &
   w(s)h^\alpha(s) \\
   0 &
   1 
   \end{pmatrix}, \; s\in \Gamma .
\end{gather}
\item $Y_n(z)$ satisfies
\begin{gather} \label{c}
Y_n(z)\begin{pmatrix}
   z^{-n} &
   0 \\
   0 &
   z^{n} 
   \end{pmatrix}
 =
I + O\left( \frac{1}{|z|} \right), \text{ as }\ |z| \rightarrow \infty.
\end{gather}
\end{subequations}
\end{enumerate}
\end{definition}

Following the arguments presented in \cite[Section 2(a)]{qRHP} we deduce that the unique solution to the $q$-RHP given by Definition \ref{qRHP def} is 
\begin{gather} \label{RHP sol}
Y_n(z) 
=
\begin{bmatrix}
   P_{n}(z) &
   \oint_{\Gamma}\frac{P_{n}(s)w(s)h^\alpha(s)}{2\pi i (z-s)}ds \\
   \gamma_{n-1}^{-1} P_{n-1}(z) &
   \oint_{\Gamma}\frac{P_{n-1}w(s)h^\alpha(s)}{ 2\pi i (z-s)\gamma_{n-1}}ds
   \end{bmatrix},
\end{gather}
where $\{P_{n}(z)\}_{n=0}^\infty$ is the family of monic $n^{th}$-degree orthogonal polynomials such that
\[ \int_{-1
}^{1} P_n(z)P_m(z)w(z)|z|^{\alpha}d_{q}z = \gamma_n\delta_{n,m} . \]

\subsection{Main results}
We are now in a position to state the main results of this paper, which are collected as Theorems \ref{main result 1} and \ref{main result 2}. The first main result concerns the asymptotic behaviour of orthogonal polynomials as their degree approaches infinity.
\begin{theorem}\label{main result 1}
Suppose that $\{P_n(z)\}_{n=0}^\infty$ is a family of monic $q$-orthogonal polynomials, orthogonal with respect to the weight $|z|^{\alpha}w(z)d_qz$, where $\alpha \in (-1,\infty)$ and $w(z)$ is an admissible weight function. Define $t = zq^{-n/2}$. Then for any given positive integer $m$ there exists an $N_m\in\mathbb N$ and $C(m)$ such that for even $n>N_m$, we have
\begin{eqnarray*}
   |(-1)^\frac{n}{2} q^{-\frac{n}{2}(\frac{n}{2}-1)}P_n(z)- \psi(t)| &<& C(m)q^n, \qquad \;\;\; \text{for } |t|\leq q^{-m-1/2}, \\
   |P_n(z)- z^{n}(z^{-2};q^2)_\infty|&<& C(m)q^{n},\qquad \; \; \, \text{for } 1\geq |z|> q^{n/2-m-1/2},
\end{eqnarray*}
and
\begin{eqnarray*}
   |(-1)^\frac{n}{2} q^{\frac{n}{2}(\frac{n}{2}-1+\alpha)}\gamma_{n-1}^{-1}P_{n-1}(z)- \varphi(t)| &<& C(m)q^n, \qquad \;\;\; \text{for } |t|\leq q^{-m-1/2}, \\
   |P_{n-1}(z)- z^{n-1}(z^{-2};q^2)_\infty|&<& C(m)q^{n},\qquad \; \; \, \text{for } 1\geq |z|> q^{n/2-m-1/2}.
\end{eqnarray*}
where $C(m)$ is a function of $m$, independent of $z,n$, and $\psi(t)$ and $\varphi(t)$, given in Definition \ref{hat tilde definition}, are entire functions independent of $n,m$ .
\end{theorem}
Our second main result concerns the asymptotic behaviour of recurrence coefficients and $L_2$ norm of $P_n$ as $n$ approaches infinity.
\begin{theorem}\label{main result 2}
Under the same hypotheses as Theorem \ref{main result 1} we have, for even $n\in\mathbb{N}$, as $n\to\infty$:
\begin{eqnarray*}
   \gamma_{n} &=& q^{n(n-1+\alpha)/2}\left( 2(q^2;q^2)_\infty^2 + O(q^n)\right),\\
   \gamma_{n-1} &=& q^{\frac{n-2}{2}(n-1+\alpha)}\left( 2(q^2;q^2)_\infty^2 + O(q^n)\right),\\
   a_n &=& q^{n -1 + \alpha}(1+O(q^n)).
\end{eqnarray*}

where $\gamma_n$ and $a_n$ are defined in Equations \eqref{0th orthogonality} and \eqref{recurrence} respectively.
\end{theorem}

\begin{remark}
Theorem \ref{main result 2} gives information about $a_n$ as $n\to \infty$, but the methodology we present in this paper does not provide a similar level of information about the asymptotic behaviour of $b_n$ as $n\to\infty$. The reason lies in the fact that the model RHP (see Section \ref{model RHP soln}) has a solution with an expansion of the form $\mathcal{W}(z) = I+\mathcal{W}^{(1)}/z+O(1/z^2)$, as $z\to\infty$, where $\mathcal{W}^{(1)}$ has zero main diagonal. This diagonal is where $b_n$ would typically appear and so our approach is only able to show that $b_n = o(q^n)$ as $n\to\infty$, without further information about the rate at which $b_n$ vanishes.
\end{remark}

\begin{remark}
Note that Theorems \ref{main result 1} and \ref{main result 2} do not require $w(z)$ to be positive in general.
\end{remark}

\begin{remark}
The case of little $q$-Jacobi polynomials \cite[Chapter 14.12]{koekoek2010hypergeometric} provides an illustration of Theorem \ref{main result 2}. This case has the orthogonality weight
\[ w(x) = |x|^{\alpha}(qx;q)_\infty/(bqx;q)_\infty, \]
and to leading order $\gamma_n$ is indeed independent of the parameter $b$ in the limit $n\to\infty$.
\end{remark}

\begin{remark} There is a number of generalisations one can make to the results in this paper that require a slight change in methodology and lead to slightly different final results.
\begin{enumerate}[label={\rm (\roman*)}]
\item In Definition \ref{admissable}, the condition on admissible weights that there exists constants $N_c$ and c such that 
\[ |1-w(\pm q^{n/2})| < cq^{n},\]
for $n>N_c$, can be relaxed to $w(0) = 1$. However, this could result in a change of the asymptotic error terms in Theorems \ref{main result 1} and \ref{main result 2} {\rm (}see the proof of Lemma \ref{JR goes to the identity}\/{\rm )}. In general, $C(m)q^n$ is the smallest upper bound achievable on the error.
\item  In Definition \ref{admissable}, the condition: $w(\pm q^k) \neq 0,\; {\rm for}\,k\in\mathbb{N}_0$ can be removed. Suppose that $w(q^j)=0$ for some $j\in\mathbb{N}_0$. Then, to compensate for this one has to change the function $f(z)$ defined in Equation \eqref{fz def} to 
\[f(z) = \prod_{j=0,j\neq i}^\infty ( 1-z^{-2}q^{2j}) .\]
\item The methodology presented in this paper can readily be extended to the Al-Salam Carlitz class of polynomials described in \cite[Chapter 18]{ismail2005classical}. In this case the discrete measure is supported on $\{q^k\} \cup \{-dq^k\}$ for $k\in\mathbb{N}_0$, where $d$ is a constant. To enable such an extension, we need new functions
\begin{eqnarray*}
h^{\alpha,d}(z) &=& \sum_{k=-\infty}^\infty \frac{q^k(1+d)z}{z^2+(d-1)zq^k-dq^{2k}} ,\\
f_d(z) &=& (z^{-1},q)_\infty(-dz^{-1},q)_\infty ,
\end{eqnarray*}
that replace $h^\alpha(z)$ and $f(z)$ respectively.
Furthermore, the $q$-difference equation
\begin{gather}\nonumber
S_A(qt) = \begin{bmatrix}
   1 &
   -t \\
   d^{-1}tq^{2-\alpha} &
   q^{-\alpha}-d^{-1}t^{2}q^{2-\alpha}
   \end{bmatrix}S_A(t),
\end{gather}
should be used instead of Equation \eqref{0 difference a}. Repeating the methodology presented in this paper with these substitutions leads to similar asymptotic estimates.

\end{enumerate}
\end{remark}
\subsection{Outline}
The paper is structured as follows. In Section \ref{derivation section} we make a series of transformations to the the $q$-RHP given in Definition \ref{qRHP def}. This motivates the form of a model RHP by taking the limit $n\rightarrow \infty$ of the RHP defined by Equation \eqref{Y_3 rhp}. Consequently in Section \ref{section 2.2}, we prove that the solution of the $q$-RHP approaches the solution to the model RHP and use this to prove Theorems \ref{main result 1} and \ref{main result 2}. In Section \ref{model RHP soln}, using $q$-difference calculus we show that there exists a unique solution to the model RHP and determine its form. In Appendix \ref{appB} we prove important properties about the solution to the model RHP. In Appendix \ref{q diff motivate} we motivate some of the arguments presented in this paper using discrete $q$-Hermite \rm{I} polynomials as an example. In Appendix \ref{Properties of hq section} we prove certain properties of $h^0(z)$ required in Section \ref{model RHP soln}.

\section{Proofs of main results}\label{proofs of main res}
In this section, we provide the proofs of Theorems \ref{main result 1} and \ref{main result 2}. To carry out the proofs, we rely on a sequence of transformations to the RHP described in Definition \ref{qRHP def}. This sequence ends with a limiting RHP, referred to as a {\it model RHP}, which is studied further in Section \ref{section 2.2} to deduce our main results. 

\subsection{Deriving the model RHP as $n\rightarrow \infty$, for even $n$}\label{derivation section}
We make a series of transformations to the RHP given by Definition \ref{qRHP def}. Recall $Y_n$ (see Equation \eqref{RHP sol}) is the solution of this RHP. Inspired by a similar approach first described by Deift \textit{et al.}, we make the series of transformations
\[ Y_n \rightarrow U_n \rightarrow V_n \rightarrow W_n, \]
which will enable us to deduce a model RHP governing $\mathcal{W}$, such that $W_n \rightarrow \mathcal{W}$ as $n\rightarrow \infty$. 

We will use the functions $f$ and $g$ in the sequence of transformations, where 
\begin{equation}\label{fz def}
    f(z) = ( z^{-2};q^2)_\infty,
\end{equation} 
and
\begin{equation}\label{g}
    g(z) = (q^2z^2,z^{-2};q^2)_\infty .
\end{equation} 
It can verified by direct calculation that $f(z)$ satisfies the $q$-difference equation:
\begin{equation}\label{fz}
    f(qz) = \left(1-\frac{1}{q^{2}z^{2}}\right)f(z),
\end{equation}
and $g(z)$ satisfies the $q$-difference equation:
\begin{equation}\label{gqdiff}
    g(qz) = -q^{-2}z^{-2}g(z).
\end{equation}

By induction, using Equation \eqref{fz}, we find for even $n$
\begin{equation}\nonumber
    (q^{n/2}z)^{n}f(q^{n/2}z) = q^{\frac{n}{2}(\frac{n}{2}-1)}f(z)\prod_{j=1}^{n/2}\left(q^{2j}z^{2}-1\right) .
\end{equation}
To be concise, let us define $g_{n}(z)$ as
\begin{equation}\label{gn} 
g_{n}(z) = f(z)\prod_{j=1}^{n/2}\left(1-q^{2j}z^{2}\right) .
\end{equation}
Thus, 
\begin{equation}\label{const}
    (q^{n/2}z)^{n}f(q^{n/2}z) = i^{n}\,q^{\frac{n}{2}(\frac{n}{2}-1)}g_n(z) .
\end{equation} 
Furthermore,
\begin{equation}\nonumber
    g(z)/g_n(z) = (q^{n+2}z^2;q^2)_{\infty},
\end{equation} 
it follows that for a fixed $z$, $g_n(z) \to g(z)$ as $n\to \infty$. 

The transformations consist of four steps.

\begin{enumerate}[label={\Roman*.},ref=\Roman*]
    \item We first define:
         \begin{gather} \nonumber
     U_n(z)=
     \begin{cases}
     Y_{n}(z)
      \begin{bmatrix}
       1 &
       0 \\
       0 &
       f(z) 
       \end{bmatrix},\ \text{for}\,z\in \text{ext}(\Gamma),\\
    Y_{n}(z), \qquad \qquad \; \; \; \text{for}\,z \in \text{int}(\Gamma).\\   
    \end{cases}
    \end{gather}
    Note that the zeros of $f(z)$ cancel with the simple poles of $h^\alpha(z)$, at $\pm q^k$ for $k \in \mathbb{N}_0$. This allows us to deform the contour, $\Gamma$, so that the poles of $h(z)$ at $\pm q^k$ for $k \in \mathbb{N}_0$ can lie in ext($\Gamma$) without affecting analyticity of the solution (see \cite[Section 2(a)]{qRHP} for a description of the holomorphicity of $Y_n$). We observe that $f(z)$ does not change the asymptotic condition; see Equation \eqref{c}.
    
    \item We now scale the contour $\Gamma$ so that the modulus of points on it are multiplied by $q^{n/2}$. (If $\Gamma$ were the unit circle, it would now be a circle with radius $q^{n/2}$.) Denote the new contour by $\Gamma_{q^{n/2}}$.
    
    \item The next transformation is 
     \begin{gather} \nonumber
     V_n(z)=
     \begin{cases}
     U_{n}(z)
      \begin{bmatrix}
       f(z)^{-1} &
       0 \\
       0 &
       1
       \end{bmatrix},\ \text{for}\,z\in \text{ext}(\Gamma_{q^{n/2}}),\\
    U_{n}(z), \qquad \qquad \qquad\ \text{for}\,z \in \text{int}(\Gamma_{q^{n/2}}).\\   
    \end{cases}
    \end{gather}
    Note we have now introduced simple poles at $\pm q^{k}$ for $k = 0,1,2,...,n/2-1$.
    \item Our final transformation is
    \begin{align}\label{wnz def} 
     W_n(z)
     = \left\{\begin{array}{lr}
     \begin{bmatrix}
       c_n^{n} &
       0 \\
       0 &
       c_n^{-n} 
       \end{bmatrix}
     V_n(z)
      \begin{bmatrix}
       (zc_n)^{-n} &
       0 \\
       0 &
       (zc_n)^n
       \end{bmatrix} &\text{for}\,z \in \text{ext}(\Gamma_{q^{n/2}}) \\
     \begin{bmatrix}
       c_n^{n} &
       0 \\
       0 &
       c_n^{-n} 
       \end{bmatrix}
      V_n(z)\qquad &\text{for}\, z \in \text{int}(\Gamma_{q^{n/2}}),
      \end{array}\right.
    \end{align}
where $c_n$ is a constant, to be defined shortly. We observe that after these transformations $W_n$ has the asymptotic condition
\begin{equation*}
    W_n(z) = I + O\left(\frac{1}{|z|}\right) .
\end{equation*}
Motivated by the form of Equation \eqref{const} we set
\[ c_n = -iq^{-\frac{1}{2}(\frac{n}{2}-1)}.\] 
\end{enumerate}

After these transformations we are left with the following transformed RHP for which the $2\times2$ complex matrix function $W_n(z)$, defined in Equation \eqref{wnz def}, is the solution:
\begin{enumerate}[label={{\rm (\roman *)}}]
\begin{subequations}\label{wn z}
\item  $W_n(z)$ is meromorphic in $\mathbb{C}\setminus \Gamma_{q^{n/2}}$ with simple poles at $z=\pm q^{k}$ for $k=0,1,2,...,n/2-1$.
\item $W_n(z)$ has continuous boundary values $W_n^-(s)$ and $W_n^+(s)$ as $z$ approaches $s\in\Gamma_{q^{n/2}}$ from $\mathcal D_{-,q^{n/2}}$ and $\mathcal D_{+,q^{n/2}}$ respectively, where
    \begin{equation} 
   W_n^+(z)
     =
      W_n^-(z)
      \begin{bmatrix}
       g_n(sq^{-n/2})^{-1} &
       q^{\alpha n/2}h^\alpha(sq^{-n/2})w(s) g_n(sq^{-n/2}) \\
       0 &
        g_n(sq^{-n/2})
       \end{bmatrix}
\end{equation}
for $\ s\in \Gamma_{q^{n/2}}$.
Note, we have used Equations \eqref{h alpha diff} and \eqref{const} to evaluate the jump condition.
\item $W_n(z)$ satisfies
    \begin{equation}
         W_n(z) = I + O\left( \frac{1}{|z|} \right), \text{ as $|z| \rightarrow \infty$} .
    \end{equation} 
\item The residue at each pole $z=\pm q^{k}$ for $k=0,1,2,...,n/2-1$, is given by
    \begin{gather} 
    \text{Res}(W_n(\pm q^{k}))
     =
      \lim_{z \rightarrow \pm q^{k}} W_n(z)
      \begin{bmatrix}
       0 &
       0 \\
       (z\mp q^k)g_{n}(zq^{-n/2})^{-2}h^\alpha(z)^{-1}w(z)^{-1} &
       0
       \end{bmatrix}.
    \end{gather}
\end{subequations} 
\end{enumerate}

We now make the change in variables $tq^{n/2} = z$. Furthermore, we scale the orthogonal weight, $|z|^\alpha w(z)$, by $q^{-\alpha n/2}$. It follows from the RHP above for $W_n(z)$ that $W_{n,scal}(t) = W_n(z) = W_n(tq^{n/2})$ solves the following RHP:
\begin{enumerate}[label={{\rm (\roman *)}}]
\begin{subequations}\label{Y_3 rhp}
\item  $W_{n,scal}(t)$ is meromorphic in $\mathbb{C}\setminus \Gamma \rightarrow \mathbb{C}^{2\times 2}$, with simple poles at $t=\pm q^{-k}$ for $k=1,2,...,n/2$.
\item $W_{n,scal}(t)$ has continuous boundary values $W_{n,scal}^-(s)$ and $W_{n,scal}^+(s)$ as $t$ approaches $s\in\Gamma$ from $\mathcal D_-$ and $\mathcal D_+$ respectively, where
    \begin{gather} 
     W_{n,scal}^+(t)
     =
      W_{n,scal}^-(t)
      \begin{bmatrix}
       g_n(s)^{-1} &
       h^{\alpha}(s)w(sq^{n/2}) g_n(s) \\
       0 &
       g_n(s)
       \end{bmatrix}, \; s\in \Gamma,
\end{gather}
\item $W_{n,scal}(t)$ satisfies
    \begin{equation}
         W_{n,scal}(t) = I + O\left( \frac{1}{|t|} \right), \text{ as $|t| \rightarrow \infty$} .
    \end{equation} 
\item The residue at each pole $t=\pm q^{-k}$ for $k=1,2,...,n/2$, is given by
    \begin{gather} \label{Y3 res}
    \text{Res}(W_{n,scal}(\pm q^{k}))
     =
      \lim_{t \rightarrow \pm q^{k}} W_{n,scal}(t)
      \begin{bmatrix}
       0 &
       0 \\
       (t\mp q^k)g_{n}(t)^{-2}h^\alpha(t)^{-1}w(tq^{n/2})^{-1} &
       0
       \end{bmatrix}.
    \end{gather}
\end{subequations} 
\end{enumerate}

As seen in the statement of Theorems \ref{main result 1} and \ref{main result 2}, we are interested in orthogonal weights which satisfy $w(zq^{n/2}) \rightarrow 1$ as $n \rightarrow \infty$. Taking the limit $n\rightarrow \infty$ of the RHP for $W_{n,scal}(t)$, motivates the following model RHP. 

\begin{definition}[Model RHP]
Assume that the contour $\Gamma$ and regions $\mathcal D_\pm$ satisfy the conditions of Definition \ref{admissable}.  
\begin{enumerate}[label={{\rm (\roman *)}}]
\begin{subequations}\label{new rhp}
\item $\mathcal{W}(t)$ is meromorphic in $\mathbb{C}\setminus \Gamma$,  with simple poles at $t=\pm q^{-k}$ for $k\in\mathbb{N}_1$.
\item $\mathcal{W}(t)$ has continuous boundary values $\mathcal{W}^-(s)$ and $\mathcal{W}^+(s)$ as $t$ approaches $s\in\Gamma$ from $\mathcal D_-$ and $\mathcal D_+$ respectively, where
    \begin{gather} \label{new jump}
      \mathcal{W}^+(t)
     =
      \mathcal{W}^-(t)
      \begin{bmatrix}
       g(s)^{-1} &
       g(s)h^\alpha(s) \\
       0 &
       g(s)
       \end{bmatrix}, \; s\in \Gamma,
    \end{gather}
    
\item $\mathcal{W}(t)$ satisfies
    \begin{equation}\label{Yinfty decay}
         \mathcal{W}(t) = I + O\left( \frac{1}{|t|} \right), \text{ as $|t| \rightarrow \infty$} .
    \end{equation} 
    However we also have that $\mathcal{W}(q^{-k})$ has poles in the LHS column for $k\in \mathbb{N}_1$. Thus, the decay condition does not hold near these poles. For example the decay condition holds for $t$ such that $|t \pm q^{-k}| > r$, for all $k\in \mathbb{N}_1$, for fixed $r>0$.
    
\item The residue at the poles $t=\pm q^{-k}$ for $k\in \mathbb{N}_1$ is given by
    \begin{gather} \label{res cond}
    \text{Res}(\mathcal{W}(\pm q^{-k}))
     =
      \lim_{t \rightarrow \pm q^{-k}} \mathcal{W}(t)
      \begin{bmatrix}
       0 &
       0 \\
       (t\mp q^{-k})g(t)^{-2}h^\alpha(t)^{-1}  &
       0
       \end{bmatrix}.
    \end{gather}
\end{subequations} 
\end{enumerate} 
\end{definition}
In Section \ref{model RHP soln} we prove that there exists a unique solution to the model RHP. We now show that $W_n(t) \rightarrow \mathcal{W}(t)$ in the limit $n\rightarrow \infty$.

\begin{remark}
In Section \ref{model RHP soln} we show $\mathcal{W}(t)$ restricted to $t\in \mathcal{D}^-$ can be analytically extended for $t\in \mathcal{D}^+$ to a $2\times2$ matrix with entire entries. Let us denote this function as $\widehat{\mathcal{W}}(t)$ {\rm (}note that $\widehat{\mathcal{W}}(t)$ = $\mathcal{W}(t)$ for $t\in \mathcal{D}^-${\rm )}. We also show that $\mathcal{W}(t)$, restricted to $t\in \mathcal{D}^+$, can be meromorphically extended for $t\in \mathcal{D}^-\setminus 0$ to a $2\times2$ matrix with simple poles for entries in the LHS column at $t=\pm q^{k}$, for $k \in \pm \mathbb{N}$. Let us denote this function as $\widetilde{\mathcal{W}}(t)$ {\rm (}note that $\widetilde{\mathcal{W}}(t)$ = $\mathcal{W}(t)$ for $t\in \mathcal{D}^+${\rm )}. For all $t\in \mathbb{C}\setminus 0$ the identity
\begin{gather}\nonumber 
     \widetilde{\mathcal{W}}(t)
     =
      \widehat{\mathcal{W}}(t)
      \begin{bmatrix}
       g(t)^{-1} &
       h^{\alpha}(t) g(t) \\
       0 &
       g(t)
       \end{bmatrix},
\end{gather}
holds. An analogous statement holds for $W_n(t)$ such that 
\begin{gather} \nonumber 
     \widetilde{W}_n(t)
     =
      \widehat{W}_n(t)
      \begin{bmatrix}
       g_n(t)^{-1} &
       h^{\alpha}(t)w(tq^{n/2}) g_n(t) \\
       0 &
       g_n(t)
       \end{bmatrix}.
\end{gather}
\end{remark}

\begin{definition}\label{hat tilde definition}
We define $\psi(t)$, $\phi(t)$, $\varphi(t)$  and $\varrho(t)$ as the $(1,1)$, $(1,2)$, $(2,1)$ and $(2,2)$ entries of $\widehat{\mathcal{W}}$ respectively.
\begin{gather} \nonumber 
     \widehat{\mathcal{W}}(t)
     =
      \begin{bmatrix}
       \psi(t) &
       \phi(t) \\
       \varphi(t) &
       \varrho(t)
       \end{bmatrix}.
\end{gather}
In Section \ref{model RHP soln} we show that these four functions are entire and can be explicitly written in terms of a power series about 0.
\end{definition}

\subsection{Proofs of Theorems \ref{main result 1} and \ref{main result 2}}\label{section 2.2}
We prove Theorem \ref{main result 1} and Theorem \ref{main result 2}. First, Theorem \ref{main result 1} is proved by showing that ${W}_n(t) \rightarrow {\mathcal{W}}(t)$ as $n \rightarrow \infty$. To do this we will construct a RHP, given by Equation \eqref{R rhp state}, that has the unique solution $R(t)$ such that $R(t) = \widetilde{W}_n(t)(\widetilde{\mathcal{W}}(t))^{-1}$ for $R(t)|_{t\in \text{ext}(\Gamma_R)}$, where $\Gamma_R$ is defined shortly. We will show that $R(t)\to I$ as $n\to \infty$, Theorem \ref{main result 1} and Theorem \ref{main result 2} then follow immediately. 

Before stating the RHP for $R(t)$ we define a number of identities.
\begin{definition}\label{gamma definition}
Define the piece-wise Jordan curve $\Gamma_R$ as $\Gamma_R = \Gamma_1 \cup \Gamma_2 \cup \Gamma_3$, where:
\begin{itemize}
    \item $\Gamma_1 = \partial B(0,q^{-m-1/2})$, where $m\in \mathbb{N}$ is a free parameter, which will later be restricted in Theorem \ref{R close to identity theorem}
    \item $\Gamma_2 = \bigcup_{k = m+1}^{n/2-1} \partial B(\pm q^{-k},r)$, where there is a large degree of freedom in choosing $r$. It is sufficient to choose $r$ such that the contours do not intersect and the orthogonality weight $w(tq^{n/2})$ is analytic in $\rm{int}(\Gamma_2)$.
    \item $\Gamma_3 = \bigcup_{k = n/2}^{\infty} \partial B(\pm q^{-k},r)$, where there is a large degree of freedom in choosing $r$. It is sufficient to choose $r$ such that the contours do not intersect and the orthogonality weight $w(tq^{n/2})$ is analytic in $\rm{int}(\Gamma_3)$.
\end{itemize}
See Figure \ref{W RHP fig} for an illustration of $\Gamma_R$. Note that $\Gamma_3$ is composed of an infinite union of circles whilst $\Gamma_2$ is composed of a finite union $(n/2-m-1)$.
\end{definition}

\begin{figure}[b]
\caption{RHP for $R$ in the $t$ plane. $\Gamma_R$ is the union of the solid black, dashed red and dotted blue circles. $\Gamma_1$ is delineated by the solid black circle, across which the jump $\widetilde{R} = \widehat{R}_1J_1$ holds. $\Gamma_2$ is delineated by the dashed red circles, across which the jump $\widetilde{R} = \widehat{R}_2J_2$ holds.  $\Gamma_3$ is delineated by the dotted blue circles, across which the jump $\widetilde{R} = \widehat{R}_3J_3$ holds.}\label{W RHP fig}
\centering
\begin{tikzpicture}
\draw (-5.5,0) -- (-2.8,0);
\draw (-2.5,0) node[left]  {\large $\wr$};
\draw[-] (-2.6,0) -- (2.6,0);
\draw (2.5,0) node[right]  {\large $\wr$};
\draw[-latex] (2.8,0) -- (5.5,0);

\draw[latex-latex] (0,-3) -- (0,3);
\draw (0,3) node[left] {\large $\Im(t)$};
\draw (6,0) node[above] {\large $\Re(t)$};

\foreach \i in {-0.0625,-0.125,-0.25,-0.5, -1, -2,-3.3,0,0.0625,0.125,0.25,0.5, 1, 2,3.3,-4.8,4.8}
\fill[black] (\i,0) circle (0.4 mm);

\draw (-2,0.5) node[above] {\small $-q^{-m-2}$};
\draw (2,0.27) node[above] {\small $q^{-m-2}$};
\draw (-3.3,0.25) node[above] {\small $-q^{-n/2}$};
\draw (3.3,0.25) node[above] {\small $q^{-n/2}$};
\draw (-1,0.25) node[above] {\small $-q^{-m-1}$};
\draw (1,0.25) node[above] {\small $\,q^{-m-1}$};
\draw (-4.8,0.25) node[above] {\small $-q^{-n/2-1}$};
\draw (4.8,0.25) node[above] {\small $\,q^{-n/2-1}$};

\draw[thin,solid] (0,0) circle (0.6cm);
\draw (0.205,0.564) -- (0.5,1) node[above right] {\small $|t| = q^{-m-1/2}$};

\draw[thin,densely dashed,red] (1,0) circle (0.3cm);
\draw[thin,densely dashed,red] (-1,0) circle (0.3cm);
\draw[thin,densely dashed,red] (2,0) circle (0.3cm);
\draw[thin,densely dashed,red] (-2,0) circle (0.3cm);
\draw[thin,densely dashed,red] (3.3,0) circle (0.3cm);
\draw[thin,densely dashed,red] (-3.3,0) circle (0.3cm);

\draw[thick,loosely dotted,blue] (4.8,0) circle (0.3cm);
\draw[thick,loosely dotted,blue] (-4.8,0) circle (0.3cm);

\draw [decorate,decoration={brace,amplitude=10pt,mirror},xshift=0pt,yshift=4pt]
(0.7,-0.5) -- (3.6,-0.5) node [black,midway,yshift=0.6cm] 
{};

\draw [decorate,decoration={brace,amplitude=10pt},xshift=0pt,yshift=4pt]
(-0.7,-0.5) -- (-3.6,-0.5) node [black,midway,yshift=0.6cm] 
{};

\draw (-2,2) node[above] {\large $\widetilde{R}$};
\draw (4.8,-0.25) node[below] {$\Gamma_3$};
\draw (-4.8,-0.25) node[below] {$\Gamma_3$};
\draw (2.1,-0.6) node[below] {$\Gamma_2$};
\draw (-2,-0.6) node[below] {$\Gamma_2$};
\draw (0.5,-0.5) node[below] {$\Gamma_1$};

\end{tikzpicture}  
\end{figure}

\begin{definition}
Define the three matrix functions:
\begin{subequations}
 \begin{align}
     J_1(t) &= \widehat{\mathcal{W}}(\widehat{W}_n)^{-1}\widetilde{W}_n(\widetilde{\mathcal{W}})^{-1}  \label{jw1} \\
     J_2(t) &= 
     \begin{bmatrix}
       1 &
       0\\
       g_n(t)^{-2}h^\alpha(t)^{-1}w(tq^{n/2})^{-1} &
       1
     \end{bmatrix}(\widetilde{\mathcal{W}})^{-1} , \label{jw2}\\
     J_3(t) &= (\widetilde{\mathcal{W}})^{-1}.\label{jw3}
\end{align}
\end{subequations}
In general these matrices have meromorphic entries with simple poles at $t = \pm q^k$, for $k\in \pm \mathbb{N}$.
\end{definition}

We now prove the following lemma.
\begin{lemma}
The unique solution to the RHP:
\begin{enumerate}[label={{\rm (\roman *)}}]
\begin{subequations}\label{R rhp state}
\item $R(t)$ is analytic in $\mathbb{C}\setminus \Gamma_R$, where $\Gamma_R$ is described above and illustrated in Figure \ref{W RHP fig},
\item $R(t)$ satisfies
    \begin{gather} 
     \lim_{t \rightarrow \Gamma_R^+} R(t)
     =
      \lim_{t \rightarrow \Gamma_R^{-}} R(t)J_R(s), \; s\in \Gamma_R,
    \end{gather}
    where $\left.J_R\right|_{\Gamma_1} = J_1|_{\Gamma_1}$, $\left.J_R\right|_{\Gamma_2} = J_2|_{\Gamma_2}$ and  $\left.J_R\right|_{\Gamma_3} = J_3|_{\Gamma_3}$,
\item 
    \begin{equation}\label{R decay}
         R(t) = I + O\left( \frac{1}{|t|} \right), \text{ as $|t| \rightarrow \infty$} ,
    \end{equation} 
\end{subequations}    
\end{enumerate} 
is given by,
\begin{equation}\label{R soln}
    R(t) = \left\{\begin{array}{lr}
       \widetilde{R} = \widetilde{W}_n(\widetilde{\mathcal{W}})^{-1} & \text{for } \left.R\right|_{\rm{ext}({\Gamma_R})},\\
        \widehat{R}_1 =  \widetilde{W}_n(\widetilde{\mathcal{W}})^{-1}(J_1)^{-1} = \widehat{W}_n(\widehat{\mathcal{W}})^{-1} & \text{for } \left.R\right|_{\rm{int}({\Gamma_1})}, \\
        \widehat{R}_2 =  \widetilde{W}_n(\widetilde{\mathcal{W}})^{-1}(J_2)^{-1} & \text{for } \left.R\right|_{\rm{int}({\Gamma_2})}, \\
        \widehat{R}_3 =  \widetilde{W}_n(\widetilde{\mathcal{W}})^{-1}(J_3)^{-1} =\widetilde{W}_n & \text{for } \left.R\right|_{\rm{int}({\Gamma_3})}.
        \end{array}\right.
\end{equation}
\end{lemma}

\begin{proof}
\textbf{Existence.} By definition $\widetilde{R}(t)$ is analytic in ext($\Gamma_R$) (as $\widetilde{W}_n$ and $\widetilde{\mathcal{W}}$ are analytic in ext($\Gamma_R$)). Thus, $R(t)$, as defined in Equation \eqref{R soln}, is analytic in ext($\Gamma_R$). We are left to show that $R(t)$ (given by Equation \eqref{R soln}) is analytic in int($\Gamma_R$).\\ 

First we look at the region $|t| < q^{-m-1/2}$. The matrix $J_1$ is defined in Equation \eqref{jw1} as
\begin{gather}\nonumber
J_1 = \widehat{\mathcal{W}}(\widehat{W}_n)^{-1}\widetilde{W}_n(\widetilde{\mathcal{W}})^{-1} = \widehat{\mathcal{W}}\begin{bmatrix}
   1/g_n &
   g_n h^\alpha w(tq^{n/2}) \\
   0 &
   g_n
   \end{bmatrix}
   \begin{bmatrix}
   g &
   -g h^\alpha \\
   0 &
   1/g
   \end{bmatrix}(\widehat{\mathcal{W}})^{-1}.
\end{gather}
By definition, $\widehat{R}_1 = \widetilde{W}_n(\widetilde{\mathcal{W}})^{-1}(J_1)^{-1} = \widehat{W}_n(\widehat{\mathcal{W}})^{-1}$ which is analytic in int($\Gamma_R$). 

Next we look at $q^{-m-1/2} < |t| < q^{-n/2}$. By definition
\begin{gather*} 
\begin{aligned}
\widehat{R}_2 &= \widetilde{W}_n(\widetilde{\mathcal{W}})^{-1}(J_2)^{-1},\\
&= \widetilde{W}_n(\widetilde{\mathcal{W}})^{-1}(\widetilde{\mathcal{W}}) \begin{bmatrix}
   1 &
   0\\
   -g_n(t)^{-2}h^\alpha(t)^{-1}w(tq^{n/2})^{-1} &
   1
   \end{bmatrix},\\
    &= \widetilde{W}_n\begin{bmatrix}
   1 &
   0\\
   -g_n(t)^{-2}h^\alpha(t)^{-1}w(tq^{n/2})^{-1} &
   1
   \end{bmatrix}.
   \end{aligned}
\end{gather*}
From the residue condition for $W_n$, given by Equation \eqref{Y3 res}, we conclude that $\widehat{R}_2$ is analytic in int($\Gamma_2$). \\ 

Finally we consider $|t|>q^{-n/2}$. By definition
\begin{eqnarray*}
    \widehat{R}_3 &=& \widetilde{W}_n(\widetilde{\mathcal{W}})^{-1}(J_3)^{-1},\\
     &=& \widetilde{W}_n(\widetilde{\mathcal{W}})^{-1}\widetilde{\mathcal{W}},\\
     &=& \widetilde{W}_n.
\end{eqnarray*}
From the residue condition for $W_n$, given by Equation \eqref{Y3 res}, we know that $\widetilde{W}_n(t)$ is analytic for $|t| > q^{-n/2}$ and thus $\widehat{R}_3$ is analytic in int($\Gamma_3$). \\

\textbf{Uniqueness.} We note that $\text{det}(W_n) = \text{det}(\mathcal{W}) = 1$. It follows that $\text{det}(J_R)=1$, and applying the same arguments as in Section \ref{uniqueness section} we conclude that if a solution exists to the RHP given by Equation \eqref{R rhp state}, then it is unique.

\end{proof}

We now prove that under certain conditions the solution, $R(t)$, to the RHP given by Equation \eqref{R rhp state} approaches the identity. We first prove a lemma about the jump matrix $J_R$. 

\begin{lemma}\label{JR goes to the identity}
    There exists an $M$ such that for any fixed integer $m>M$, the jump conditions $J_i$ defined in Equations \eqref{jw1}, \eqref{jw2} and \eqref{jw3} satisfy:
\begin{eqnarray}
    \Vert J_1(t) - I\Vert_{\Gamma_1} &=& C(m)q^n,\\
    \Vert J_2(t) - I\Vert_{\Gamma_2} &<& 1/2,\\
    \Vert J_3(t) - I\Vert_{\Gamma_3} &=& O(q^{n/2}),
\end{eqnarray}
for large $n>N_m$. {\rm (}Where $\Vert J_1(t) - I\Vert_{\Gamma_1}$ is the infinity norm of the matrix $J_1(t) - I$ restricted to the curve $\Gamma_1$, and $C(m)$ is a function of $m$, independent of $n,t$.{\rm )}
\end{lemma}

\begin{proof}
First we consider $J_3$. By the asymptotic condition, Equation \eqref{Yinfty decay}, we know that 
\begin{equation}\label{step 1 eq}
    \Vert \widetilde{\mathcal{W}}(t)-I\Vert_{\partial B(\pm q^{-k},r)} < C_1q^{k},
\end{equation}
where the radius $r$ can be ignored as we will be considering the case $r\ll q^{-k}$. Applying Equation \eqref{step 1 eq} and Definition \ref{gamma definition} we find 
\[\Vert \widetilde{\mathcal{W}}(t)-I\Vert_{\Gamma_3} < C_1q^{n/2}. \]
By definition $J_3(t) = \widetilde{\mathcal{W}}(t)$, hence we conclude
\[\Vert J_3(t)-I\Vert_{\Gamma_3} < C_1q^{n/2}. \]

Next we study $J_2$. Applying Equation \eqref{step 1 eq} we find for $\Gamma_2 = \bigcup_{k = m+1}^{n/2-1} \partial B(\pm q^{-k},r)$
\begin{equation*}
\Vert \widetilde{\mathcal{W}}(t)-I\Vert_{\Gamma_2} < C_1q^{m}.
\end{equation*}
Furthermore, we observe that the matrix 
\begin{gather*} 
\Psi(t) = \begin{bmatrix}
   1 &
   0\\
   g_n(t)^{-2}h^\alpha(t)^{-1}w(tq^{n/2})^{-1} &
   1
   \end{bmatrix} - I,
\end{gather*}
vanishes much faster than $\Vert \Psi(t)-I\Vert_{\partial B(\pm q^{-k},r)} < C/|t|$. To see this, recall from Equation \eqref{gn} that 
\begin{equation*}
g_{n}(t) = \left( \prod_{j=1}^{n/2}1-q^{2j}t^{2}\right)\prod_{j=0}^{\infty}\left(1-\frac{q^{2j}}{t^{2}}\right) .
\end{equation*}
Thus, $g_n(t)^{-2}$ is vanishingly small for large $|t|$. (To see this, expand out the first few terms of the product $\prod_{j=1}^{n/2}1-q^{2j}t^{2}$). Hence we conclude that
\[ \Vert J_2(t)-I\Vert_{\Gamma_2} < C_2 q^{m}. \]
It follows there exists an $M$ such that for $m > M$.
\[ \Vert J_2(t)-I\Vert_{\Gamma_2} < 1/2. \] 

Next we study $J_1$. Define the function 
\begin{equation} 
p(t) = g(t)/g_n(t) = \prod_{j=n/2+1}^{\infty}\left(1-q^{2j}t^{2}\right),
\end{equation}
where $g_n$ and $g$ are given in Equations \eqref{gn} and \eqref{g} respectively. For large $n$, we can take a Taylor series expansion of ${\rm log}(p(t))$ to find
\begin{eqnarray}
   {\rm log}(p(t)) &=& \sum_{j=1}^\infty {\rm log}(1-q^{n+2j}t^{2}),\nonumber\\
   &<& 2 \sum_{j=1}^\infty -q^{n+2j}t^{2}, \nonumber \\
   &=& q^{n} \frac{-2q^2t^2}{1-q^2}. \label{p(t) On}
\end{eqnarray}
Furthermore, define
\[ H(t) = h(t)\left( p(t)-p(t)^{-1}w(tq^{n/2}) \right) .\]
Then, expanding Equation \eqref{jw1} we find 
\begin{eqnarray*}
J_1(1,1) &=& \psi\varrho p - \phi\varphi /p - \psi\varphi H ,\\
J_1(1,2) &=& \phi\psi(1/p - p) + \psi^{2}H ,\\
J_1(2,1) &=& \varrho\varphi(p-1/p) - \varphi^2H ,\\
J_1(2,2) &=& \psi\varrho/ p - \phi\varphi p + \psi\varphi H ,
\end{eqnarray*}
where $\psi, \phi, \varphi, \varrho$ are defined in Definition \ref{hat tilde definition}. We note that $\psi\varrho-\varphi\phi = \text{det}(\widehat{\mathcal{W}}) = 1$. Applying Equation \eqref{p(t) On} it is clear that for a fixed $|t|=q^{-m-1/2}$, there exists an $N_m$ such that for $n>N_m$
\[\Vert J_1(t)-I\Vert_{\Gamma_1} < C_3(m)q^{n}, \]
where $C_3(m)$ is a function of $m$. Note that we choose $|t|=q^{-m-1/2}$ because there exists poles of the jump function, $J_1$, at $t = \pm q^{-m}$ for integer values of $m$. 
\end{proof}

\begin{remark}\label{remark 2.11}
Lemma \ref{JR goes to the identity} holds for $t$ lying on $\Gamma_R$. However, from Equations \eqref{jw1}, \eqref{jw2} and \eqref{jw3}, the matrix functions $J_i(t):\mathbb{C}\rightarrow \mathbb{C}^{2\times2}$ are well defined for all $t\in \mathbb{C}$. In general, the matrices $J_i(t)$ do not approach the identity everywhere, but, they will have simple poles of order $O(q^n)$ at $\pm q^k$, for $k\in \pm \mathbb{N}$. Furthermore, 
\[ \sum_{k=-\infty}^{\infty} |{\rm Res}(J_i(q^k))| = O(q^n). \]
This follows by direct computation, observing that $g(t)^{-1}$ and $g_n(t)^{-1}$ are vanishingly small for large $t$, and applying Equation \eqref{bar a} which demonstrates that $\psi(q^{-k})$ and $\varphi(q^{-k})$ also become vanishingly small for large positive integer values of $k$.
\end{remark}

Having proved Lemma \ref{JR goes to the identity} we are now in a position to show that the solution, $R(t)$, to the RHP given by Equation \eqref{R rhp state} approaches the identity.
\begin{lemma}\label{R close to identity theorem}
    For a given integer $m>M$, where $M$ satisfies Lemma \ref{JR goes to the identity}, the solution, $R(t)$, to the RHP defined in Equation \eqref{R rhp state} is bounded for large $n$. Furthermore, $\widehat{R}_1(t) = 1 + O(q^n)$ and $\widetilde{R}(t) = 1 + O(q^n)$.
\end{lemma}

\begin{proof}
Define 
\begin{equation*}
    \Delta = J_{R}-I.
\end{equation*}
It immediately follows that for $s$ in $\Gamma_R$
\begin{equation*}
    \lim_{t\to\Gamma^+}R(t) = \lim_{t\to\Gamma^-}R(t)(I+\Delta(s)).
\end{equation*}
By the asymptotic condition, Equation \eqref{R decay}, we conclude that
\begin{eqnarray}
    \widetilde{R}(t) &=& I +\frac{1}{2\pi i}\oint_{\Gamma_R}\frac{R(s)\Delta(s)}{t-s}ds, \label{R outside epression} \\
    &=& I + \sum_{k=-\infty}^{\infty}\frac{\text{Res}(R(\pm q^{k})\Delta(\pm q^k))}{t\pm q^k}, \nonumber\\
     &=& I + \sum_{k=-m}^{\infty}\frac{\text{Res}(\widehat{R}_1(\pm q^{k})J_1(\pm q^k))}{t\pm q^k} \nonumber\\
     &&\; + \sum_{k=-n/2}^{-m-1}\frac{\text{Res}(\widehat{R}_2(\pm q^{k})J_2(\pm q^k))}{t\pm q^k}+ \sum_{k=-\infty}^{-n/2-1}\frac{\text{Res}(\widehat{R}_3(\pm q^{k})J_3(\pm q^k))}{t\pm q^k}. \nonumber
\end{eqnarray}
Let $L$ be defined as $L = \sup_{t\in\text{int}(\Gamma_R)}(|R(t)|)$. As $R(t)$ is analytic in $\text{int}(\Gamma_R)$ it follows $|R(t)|$ achieves its maximum on the boundary (i.e. on $\Gamma_R$). Therefore,
\begin{equation*}
    L = \left| \left(I + \sum_{k=-\infty}^{\infty}\frac{\text{Res}(R(\pm q^{k})\Delta(\pm q^k))}{s\pm q^k}\right)\left(I+\Delta(s)\right)^{-1}\right|,
\end{equation*}
for some $s$ on $\Gamma_R$. Furthermore, $L > |R(\pm q^k)|$ for $k \in \pm \mathbb{N}$, as the points $t=\pm q^k$ lie in $\text{int}(\Gamma_R)$ for $k \in \pm \mathbb{N}$. As $\Vert \Delta\Vert_\infty<1/2$, we can also determine $\left(I+\Delta(s)\right)^{-1}$ using the Neumann series 
\begin{equation*}
    \left(I+\Delta(s)\right)^{-1} = \sum_{j=0}^{\infty}(-\Delta(s))^{j}.
\end{equation*}
Thus we find that,
\begin{eqnarray}
    L &<& \left(I + L \sum_{k=-\infty}^{\infty}\left|\frac{\text{Res}(\Delta(\pm q^k))}{s\pm q^k}\right| \right)\sum_{j=0}^{\infty}\Vert \Delta(s)\Vert^{j},\nonumber\\
    &<& 2\left(I + L \sum_{k=-\infty}^{\infty}\left|\frac{\text{Res}(\Delta(\pm q^k))}{s\pm q^k}\right| \right). \label{M ineq}
\end{eqnarray}
It follows from Remark \ref{remark 2.11} the sum on the RHS of Equation \eqref{M ineq} converges and
\[
\sum_{k=-\infty}^{\infty}\left|\frac{\text{Res}(\Delta(\pm q^k))}{s\pm q^k}\right| = O(q^n).\]
Thus, Equation \eqref{M ineq} gives:
\begin{equation}\label{the end}
    L < 2 + 2 L \times O(q^n).
\end{equation}
It follows there exists an $N$ such that for $n>N$, $L < 2$. Hence, we have just determined an upper bound for $|R(t)|$ inside int($\Gamma_R$). Observing that $\text{Res}(\Delta(\pm q^k)) = O(q^n)$, Equation \eqref{R outside epression} implies that for any fixed radius $r>0$
\begin{equation}\label{tildeR Oq}
    \widetilde{R}(t) = I + O(q^n)/r \; \text{for }|t \pm q^{\pm k}| > r. 
\end{equation}
By definition for $s$ lying on $\Gamma_1$,
\begin{equation}\label{R1 Oq}
    \widehat{R}_1(s) = \widetilde{R}(s)(J_1)^{-1}.
\end{equation} 
Recall that $|\widehat{R}_1(t)|$ achieves its maximum on $\Gamma_1$. Applying Equations \eqref{tildeR Oq} and \eqref{R1 Oq} and Lemma \ref{JR goes to the identity}, where we showed $\Vert J_1 - I\Vert  = C(m)O(q^n)$, we conclude $\widehat{R}_1(t) = I+ C(m)O(q^n)$.  
\end{proof}

Having proved Lemma \ref{R close to identity theorem} we are now in a position to prove Theorem \ref{main result 1}.

\begin{proof}[Proof of Theorem \ref{main result 1}]
For $|t| \leq q^{-m-1/2}$ the result follows immediately from Lemma \ref{R close to identity theorem}. Lemma \ref{R close to identity theorem} implies that
\[ \widehat{W}_n(z) \widehat{\mathcal{W}}(t)^{-1} = I + C(m)O(q^n), \]
for large $n$. Theorem \ref{main result 1} follows immediately after reversing the transformations $Y_n \rightarrow W_n$.\\

For $|t| > q^{-m-1/2}$, we observe that Equation \eqref{tildeR Oq} implies that 
\[ \widetilde{R}(t) - I = O(q^n), \]
for $|t \pm q^{\pm k}| > r$ (for some fixed $r$). We also observe that $\widetilde{\mathcal{W}}- I$ is bounded for $|t| > q^{-m-1/2}$ and $|t \pm q^{\pm k}| > r$, and goes to zero as $|t|\to\infty$. Furthermore, Equation \eqref{ag} implies that the poles of $\mathcal{W}$ and $R(t)$ vanish for large $|t|$, much faster than the function $f(t)$ defined in Equation \eqref{fz def} grows. As, 
\[ \widetilde{W_n}(z)\widetilde{\mathcal{W}}(t)^{-1} = R(t) ,\]
this allows one to more accurately describe the behaviour of $P_n(q^k)$, $k\in\mathbb{N}$, as $n\to\infty$.
\end{proof}

We now prove Theorem \ref{main result 2}.

\begin{proof}[Proof of Theorem \ref{main result 2}].
Theorem \ref{main result 2} follows from Lemma \ref{R close to identity theorem}. We note that in transforming from Equation \eqref{wn z} to Equation \eqref{Y_3 rhp} the weight function was scaled by $q^{-n\alpha/2}$. Let $\zeta_n = q^{-n\alpha/2}\gamma_n$. Substituting in
\begin{gather}\nonumber
Y_n(z) 
=
\begin{bmatrix}
   P_{n}(z) &
   \oint_{\Gamma}\frac{P_{n}(s)w(s)h^\alpha(s)}{2\pi i (z-s)}ds \\
   \gamma_{n-1}^{-1} P_{n-1}(z) &
   \oint_{\Gamma}\frac{ P_{n-1}(s)w(s)h^\alpha(s)}{\gamma_{n-1}2\pi i (z-s)}ds
   \end{bmatrix},
\end{gather}
we can evaluate the expression 
\begin{gather}\nonumber
\begin{bmatrix}
   c_n^n &
   0 \\
   0 &
   c_n^{-n}
   \end{bmatrix}W_n(t) \begin{bmatrix}
   (c_ntq^{n/2})^{-n} &
   0 \\
   0 &
   (c_ntq^{n/2})^n
   \end{bmatrix}
=
 \begin{bmatrix}
   1 &
   0 \\
   0 &
   1
   \end{bmatrix}
   +
\frac{1}{tq^{n/2}}\begin{bmatrix}
    \beta_{n} &
   \zeta_n c_n^{2n} \\
   \zeta_{n-1}^{-1}c_n^{-2n}  &
   \xi_n
   \end{bmatrix} + O(t^{-2}),
\end{gather}
using the transformations detailed in Section \ref{derivation section}. We note that the function $f(z)$, defined in Equation \eqref{fz}, is even and does not impact on the $z^{-1}$ coefficient during the transformations. Let,
\begin{gather}\nonumber
\mathcal{W}(t) 
=
    \begin{bmatrix}
   1 &
   0 \\
   0 &
   1
   \end{bmatrix}
   +
\frac{1}{t}\begin{bmatrix}
    0 &
   B \\
   C  &
   0
   \end{bmatrix} + O(t^{-2}).
\end{gather}
In Section \ref{model RHP soln} we show that this is a valid representation of the solution $\mathcal{W}(t)$ for large $t$. Applying Lemma \ref{R close to identity theorem} we observe that $\widetilde{W}_n = \widetilde{\mathcal{W}} + O(q^n)$. Comparing coefficients of the $t^{-1}$ power in the top right term we find in the limit $n\to \infty$:
\begin{eqnarray*}
\zeta_n (-iq^{-\frac{1}{2}(\frac{n}{2}-1)})^{2n} &=& B(1+O(q^n))q^{n/2} ,\\
\zeta_n &=& B(1+O(q^n))q^{n(\frac{n}{2}-1)}q^{n/2}, \\
 &=& B(1+O(q^n))q^{n(n-1)/2}.
\end{eqnarray*}
Similarly in the bottom left term we find in the limit $n\to\infty$:
\begin{equation*}
    \zeta_{n-1}^{-1} = C(1+O(q^n))q^{-1}q^{-(n-1)(n-2)/2},
\end{equation*}
The constants $B$ and $C$ can be evaluated by observing that generalised discrete $q$-Hermite I polynomials must satisfy these asymptotic conditions. Thus, $B=2(q^2;q^2)_\infty$ and $C=q^{1+\alpha}/B$. Theorem \ref{main result 2} follows.

\end{proof}

\section{On the existence of a unique solution to the model RHP}\label{model RHP soln}
In this section we prove that there exists a unique solution to the model RHP given by Equation \eqref{new rhp}.
\subsection{Existence}
We first show that there exists a solution to the model RHP. This is achieved by determining the connection matrix \cite{Carmichael} between three solutions of a $q$-difference equation (Equation \eqref{0 difference b}), $S_A(t) h^\alpha(t)$, $S_B(t)$ and $S_C(t)/g(t)$ (defined shortly). We prove that the connection matrix is equivalent to the jump condition, Equation \eqref{new jump}, of the model RHP. Consequently, we show that the model RHP is satisfied by $S_A(t)$, $S_B(t)$ and $S_C(t)$, after appropriate transformations. To begin, let us consider the two $q$-difference equations
\begin{subequations}
\begin{gather}\label{0 difference a} 
S_A(qt) = \begin{bmatrix}
   1 &
   -t \\
   tq^{2-\alpha} &
   q^{-\alpha}-t^{2}q^{2-\alpha}
   \end{bmatrix}S_A(t),
\end{gather}
\begin{gather}\label{0 difference b} 
S_B(qt) = \begin{bmatrix}
   q^\alpha &
   -tq^\alpha \\
   tq^{2} &
   1-t^{2}q^{2}
   \end{bmatrix}S_B(t),
\end{gather}

where $S_A(t)$ and $S_B(t)$ are vectors with complex entries and $\alpha \in (-1,\infty)$ is a real parameter. Note that: $S_B(qt)/S_B(t) = q^\alpha S_A(qt)/S_A(t)$. We motivate the form of Equation \eqref{0 difference a} in Appendix \ref{q diff motivate}. Writing the entries of $S_A(t) = [S^1_A(t), S^2_A(t)]^T$ as a power series in $t$, we find by direct substitution into Equation \eqref{0 difference a} that $S^2_A(t)$ can be written in terms of the odd power series 
\[ S_A^2(t) = A_{2,1}t + A_{2,3}t^3  + A_{2,5}t^5 + ... + A_{2,j}t^{2j+1} + ... \,,\]
where
\[ A_{2,2j+1} = \frac{-q^{3-\alpha+2j}A_{2,2j-1}}{(q^{-\alpha}-q^{2j+1})(1-q^{2j})} .\]
Likewise, $S^1_A(t)$ can be written as an even power series
\[ S_A^1(t) = A_{1,0} + A_{1,2}t^2 + A_{1,4}t^4 + ... + A_{1,l}t^{2l} + ... \,,\]
where
\[ A_{1,2l} = \frac{A_{2,2l-1}}{1-q^{2l}} .\]
From the recurrence relations we can deduce that both entries of $S_A$ are entire. To see this, observe that for $0<q<1$
\[ \lim_{j\to\infty} \frac{A_{2,2j+1}}{A_{2,2j-1}} \rightarrow 0.\]
Similarly, $S_B(t) = [S^1_B(t), S^2_B(t)]^T$ can be written in terms of power series which converge everywhere. However, in this case $S^1_B(t)$ is odd and $S^2_B(t)$ is even. 

Now consider the $q$-difference equation
\begin{gather}\label{infinity difference} 
S_C(qt) = \frac{-1}{q^2t^2}\begin{bmatrix}
   q^\alpha &
   -tq^\alpha \\
   tq^{2} &
   1-t^{2}q^{2}
   \end{bmatrix}S_C(t).
\end{gather}
\end{subequations}
 Note the similarity to Equation \eqref{0 difference b}. One can readily show that there exists a solution to Equation \eqref{infinity difference}, $S_C(t) = [S^1_C(t), S^2_C(t)]^T$ which can be represented by a power series at infinity 
 \begin{eqnarray}\nonumber
    S^1_C &=& \sum_{j=0}^{\infty} C_{1,2j+1}t^{-(2j+1)},\, C_{1,1} \neq 0,\\
    S^2_C &=& \sum_{l=0}^{\infty} C_{2,2l}t^{-2l},\, C_{2,0} \neq 0,
 \end{eqnarray}
which converges everywhere (except obviously at 0). Earlier, in Equation \eqref{g} we defined the even function $g(t)$, which satisfies 
 \[ g(t)/g(qt) = -q^2t^2.\]
We also earlier defined the function $h^{\alpha}(t)$ in Equation \eqref{h new def}, which satisfies the $q$-difference equation
\[ h^\alpha(qt)/h^\alpha(t) = q^{\alpha} .\]
As both $S_C(t)/g(t)$ and $S_A(t)h^\alpha$ satisfy the $q$-difference equation \eqref{0 difference b} we conclude that
\begin{equation}\label{connection}
    S_C(t)/g(t) = P_1(t)h^\alpha(t)S_A(t) + P_2(t)S_B(t),
\end{equation}
where $P_1(qt) = P_1(t)$ and $P_2(qt) = P_2(t)$. This is the equivalent to a column of the connection matrix in \cite{Carmichael}. As $S_C(t)/g(t)$ is a meromorphic function with simple poles at $t = \pm q^{k}$ for $k \in \pm \mathbb{N}_0$ we conclude that $P_1(t)$ must be a constant and $P_2(t)$ must be either be a constant or a meromorphic function with simple poles at $t=\pm q^k$. Thus, by Corollary \ref{even coror} we conclude
\begin{equation*}
    P_2(t) = c_1h^0(z) + c_0,
\end{equation*}
where $c_1$ and $c_0$ are constants, and $h^0(z)$ is defined in Equation \eqref{h new def}. Comparing odd and even terms in Equation \eqref{connection} we conclude that $c_1=0$ and $P_2(t) = c_0$. Thus, after absorbing constants into the power series of $S_A$ and $S_B$, we have shown,
\begin{equation}\label{s jump}
    S_C(t) = g(t)h^\alpha(t)S_A(t) + g(t)S_B(t),
\end{equation}
and 
\begin{gather}\nonumber
S_C = \begin{bmatrix}
   0 \\
    1
   \end{bmatrix} + O\left(\frac{1}{|t|}\right), \; \text{as }t\rightarrow \infty.
\end{gather}
Furthermore, in Section \ref{appB} we show that $S_A(t)/g(t)$ satisfies the asymptotic condition, 
\begin{gather}\nonumber
S_A(t)/g(t) = \begin{bmatrix}
   C_0 \\
    0
   \end{bmatrix} + O\left(\frac{1}{|t|}\right), \; \text{as }t\rightarrow \infty,
\end{gather}
where $C_0$ is a constant. Hence, in summary we have proved
\begin{gather}\label{found jump}
[S_A(t)/g(t),S_C(t)] =  [S_A(t),S_B(t)]\begin{bmatrix}
   g(t)^{-1} & h^{\alpha}(t)g(t)\\
    0 & g(t) 
   \end{bmatrix}, 
\end{gather}
where $S_A(t)$ and $S_B(t)$ are analytic everywhere and $S_C(t)$ is analytic everywhere except $t=0$. Furthermore, we have proved the asymptotic behaviour
\begin{gather}\label{found decay}
\begin{bmatrix}
   C_0^{-1} & 0\\
    0 & 1 
   \end{bmatrix}[S_A(t)/g(t),S_C(t)] = I + O\left(\frac{1}{|t|}\right), \; \text{as }t\rightarrow \infty.
\end{gather}
Thus, after appropriate scaling we have found a solution which satisfies condition $(i)$ of the model RHP: by the holomorphicity of $S_A(t)$, $S_B(t)$ and $S_C(t)$, condition $(ii)$: by Equation \eqref{found jump}, condition $(iii)$: by Equation \eqref{found decay}, and  condition $(iv)$: by Equation \eqref{found jump}.

\subsection{Uniqueness}\label{uniqueness section}
Uniqueness follows by considering the determinant of a solution, $\mathcal{W}$, to the model RHP, Equation \eqref{new rhp}. By the residue condition, Equation \eqref{res cond}, we can deduce that det($\mathcal{W}$) is analytic in $\mathbb{C}\setminus \Gamma$. Furthermore, by the jump condition, Equation \eqref{new jump}, we deduce that det($\mathcal{W}$) is entire. Applying Louiville's theorem we conclude that the asymptotic condition, Equation \eqref{Yinfty decay}, implies that $\text{det}(\mathcal{W})=1$ everywhere. 

Suppose that there exists another solution, $\mathcal{W}_2$ to the model RHP. By the residue condition, Equation \eqref{res cond}, $\mathcal{W}_2 \mathcal{W}^{-1}$ is analytic everywhere. Furthermore, the jump conditions cancel and we can conclude $\mathcal{W}_2 \mathcal{W}^{-1}$ is entire.  Applying Louiville's theorem we conclude that the asymptotic condition, Equation \eqref{Yinfty decay}, implies that $\mathcal{W}_2 \mathcal{W}^{-1} = I$. Thus, $\mathcal{W}_2 = \mathcal{W}$.
\section{Conclusion}
In this paper, we determined the asymptotic behaviour of a general class of $q$-orthogonal polynomials by using the $q$-RHP setting \cite{qRHP}. The work is motivated by the methods developed by Deift \textit{et al.} \cite{Deift1999strong}, which used the RHP setting to determine the asymptotic behaviour of semi-classical orthogonal polynomials. The main results are Theorems \ref{main result 1} and \ref{main result 2} which provide more detailed asymptotic results for a large class of $q$-orthogonal polynomials than we could find in the literature.

There are a number of observations we can make from the results of this paper. In particular we proved that $\lim_{n\rightarrow \infty}\gamma_n$ is only dependent on the $\lim_{z\to 0}w(z)$, where $w(z)d_qz$ is the orthogonality measure. Furthermore, we note that the results in this paper hold even if $w(q^k)<0$ for some $k\in\mathbb{N}_0$. That is we do not require positivity of the weight function. When determining a solution to the model RHP we observed some interesting examples of $q$-RHP theory. For example we demonstrated how to explicitly determine a connection matrix between two solutions of a $q$-difference equation and, in Appendix \ref{appB}, also saw the relationship between $q$-Borel transforms and divergent power series arising in $q$-difference equations. 

An interesting avenue for future exploration would be to extend the results of the current paper to a larger class of $q$-orthogonal polynomials. Another possible direction could be determining if the theory presented in this paper can be applied to other settings not just $q$-discrete weights and orthogonal polynomials.
\appendix

\section{Properties of the solution to the model RHP}\label{appB}

In this section we prove some important properties of the solution to the RHP given by Equation \eqref{new rhp}. These results are used in Sections \ref{proofs of main res} and \ref{model RHP soln}. The section concludes with a remark which highlights an interesting connection between the present work and $q$-Stoke's phenomena. Note this is a side observation and not necessary for the proofs of the main results of this paper. As shown in Section \ref{model RHP soln} studying the solution to the RHP given by Equation \eqref{new rhp} is equivalent to studying the solutions $S_A$, $S_B$ and $S_C$ of the $q$-difference equations given in Section \ref{model RHP soln}. 

\begin{definition}
    For conciseness, we will adopt the notation $f(qz) = \bar{f}(z)$ throughout the appendix.
\end{definition}

Let 
\begin{gather*} 
[S_A,S_B] = \begin{bmatrix}
   a & b\\
   c & d
   \end{bmatrix} ,
\end{gather*}
where $S_A(t)$ and $S_B(t)$ satisfy Equations \eqref{0 difference a} and \eqref{0 difference b} respectively. From Equation \eqref{0 difference a} it can be deduced that $a$ satisfies the second order $q$-difference equation
\begin{equation*}
    \bar{\bar{a}} + \bar{a}(t^2q^{3-\alpha}-(1+q^{1-\alpha})) + q^{1-\alpha}a = 0.
\end{equation*}
Let us consider $a$ evaluated at $t = q^{-k}$, for large integer $k$, note that these locations coincide with the poles of $h^\alpha$ and the zeros of $g$ defined in Equations \eqref{h new def} and \eqref{g} respectively.  
If 
\[\bar{\bar{a}} \ll \bar{a}t^2q^{3-\alpha},\]
then this implies that $a$ grows as $a/\bar{a} = -t^2q^2(1+o(1))$. Now consider the jump condition given in Equation \eqref{s jump}
\[ agh^\alpha + bg = S_{C}^1.\] 
At the points of interest, $t = q^{-k}$, it can readily be verified by direct calculation from Equations \eqref{gqdiff} and \eqref{h alpha diff} that $gh^\alpha$ satisfies the $q$-difference equation:
\[ \overline{gh^{\alpha}}/gh^{\alpha} = -q^{-1+\alpha}t^{-2} .\]
Note that $g(t)=0$ and $h(t)^\alpha$ has a simple pole at $t = q^{-k}$ for $k\in \mathbb{N}$. If $a$ grows as $a/\bar{a}=-t^2q^2$ then $agh^\alpha$ grows as $agh^\alpha/\overline{agh^\alpha}=t^4q^{3-\alpha}(1+o(1))$, but the first entry in $S_{C}$, $S_C^1$, decays as $1/t$ for large $t$ which is a contradiction (remembering that $bg = 0$ as $g=0$ at our points of interest).
Thus, $a$ shrinks like $\bar{\bar{a}}/\bar{a} = -t^2q^{3-\alpha}(1+o(1))$. Hence,
\begin{equation}\label{bar a}
    \bar{a}/a = -t^2q^{1-\alpha}(1 + o(1)),
\end{equation}
at $t=q^{-k}$. Note that this indicates that the residue of the poles of $a/g$ are rapidly shrinking, such that
\begin{equation}\label{ag}
    \text{Res}(\bar{a}/\bar{g})/\text{Res}(a/g) = t^4q^{4-\alpha}(1 + o(1)).
\end{equation}
Furthermore, as $a$ is an entire function by Liouville's theorem there must be a $t\in \mathbb{C}$ such that 
\[\bar{\bar{a}} \ll \bar{a}t^2q^{3-\alpha}.\]
Thus, it follows that along this ray (given by iterating $t$) in the complex plane,
\[ a/\bar{a} = -t^2q^2(1+o(1)),\]
which importantly means that 
\begin{equation} \label{combined decay}
    a\bar{g}/\bar{a}g = 1+o(1)
\end{equation}
Note that $a/g$ is a meromorphic function, of the form
\begin{equation*}
    a/g = \sum_{j=0}^\infty c_jz^{-j} + \sum_{k=-\infty}^{\infty} \frac{d_k}{t^2 - q^{2k}}
\end{equation*}
with vanishingly small poles (Equation \eqref{ag}). Hence, Equation \eqref{combined decay} implies that $c_0$ is non-zero and $a/g$ approaches a non-zero constant for large $t$. A similar argument for the second entry of $S_a$ shows that 
\begin{gather*} 
S_a/g = \begin{bmatrix}
   c_0\\
   0
   \end{bmatrix} + O\left(\frac{1}{|t|}\right), \; \text{as }t\rightarrow \infty,
\end{gather*}
for $|t\pm q^{-k}|>r$, $k\in \mathbb{N}$ (with fixed $r>0$).
\begin{remark}
By solving the $q$-discrete equation satisfied by $a/g$ one can determine a divergent power series representation for $a/g$ at infinity. Taking a $q$-Borel transformation \cite{ramis2013local} we expect this series to represent the presence of a theta function `switching', analogous to Stokes phenomena. This is reflected in the vanishing poles found in Equation \eqref{ag}.
\end{remark}

\section{$q$-Hermite $q$-difference equation}\label{q diff motivate}
In this section we motivate the form of Equation \eqref{0 difference a} by studying the $q$-difference equation satisfied by discrete $q$-Hermite \rm{I} polynomials. Discrete $q$-Hermite \rm{I} polynomials satisfy the $q$-difference equation:
\begin{gather*} 
\overline{Y}_{n}(z) = \begin{bmatrix}
   1 &
   z(q^{n}-1) \\
   zq^{2-n} &
   1-z^{2}q^{2-n}
   \end{bmatrix}Y_{n}(z),
\end{gather*}
where 
\begin{gather*}
Y_{n}(z) = \begin{bmatrix}
   P_n\\
  P_{n-1}
   \end{bmatrix}.
\end{gather*}
After making the transformation $tq^{n/2} = z$ we find
\begin{gather*} 
\overline{Y}_{n}(t) = \begin{bmatrix}
   1 &
   tq^{n/2}(q^{n}-1) \\
   tq^{2-n/2} &
   1-t^{2}q^{2}
   \end{bmatrix}Y_{n}(t).
\end{gather*}
After taking the linear transformation 
\begin{gather*} 
S_n(t) = \begin{bmatrix}
   1 &
   0 \\
   0 &
   q^{n/2}
   \end{bmatrix}Y_{n}(t).
\end{gather*}
We find that $S_n$ satisfies the $q$-difference equation
\begin{gather}
\overline{S}_{n}(t) = \begin{bmatrix}
   1 &
   t(q^n-1) \\
   tq^{2} &
   1-t^{2}q^{2}
   \end{bmatrix}S_n(t).
\end{gather}
Taking the limit $n\rightarrow \infty$ the $q$-difference equation for $S_n$ becomes,
\begin{gather}
\overline{S}_{\infty}(t) = \begin{bmatrix}
   1 &
   -t \\
   tq^{2} &
   1-t^{2}q^{2}
   \end{bmatrix}S_\infty(t).
\end{gather}
We would expect that the solution to this difference equation solves the model RHP for the case $\alpha = 0$, and indeed that is what we find. 

\section{Functions invariant under $z\mapsto qz$ }\label{Properties of hq section}
In this section we prove some properties about meromorphic functions with simple poles which are invariant under the transformation $z\mapsto qz$, i.e. $C(qz) = C(z)$.

\begin{lemma}\label{lemma 5.6}
Let $C(z)$ be a function defined on $\mathbb{C}\setminus 0$, which is analytic everywhere except for simple poles at $q^k$ for $k \in \mathbb{Z}$. Then, $C(qz) \neq C(z)$. 
\end{lemma}

\begin{proof}
We prove the result by contradiction. Assume $C(qz) = C(z)$. Define
\[ G(z) = (-z,-qz^{-1};q)_\infty .\]
By direct calculation one can show $G(qz) = z^{-1}G(z)$. Furthermore, by definition, $G(z)$ is zero on the $q$-lattice $q^k$, $k \in \mathbb{Z}$. Let
\[ F(z) = C(z)G(z),\]
then it follows $F(z)$ is analytic in $\mathbb{C}\setminus 0$ and satisfies the difference equation
\begin{equation}\label{cF diff eq}
    F(qz) = z^{-1}F(z) .
\end{equation}
As $F(z)$ is analytic in $\mathbb{C}\setminus 0$ we can write $F(z)$ as the Laurent series
\[ F(z) = \sum_{k=-\infty}^\infty F_kz^k.\]
Comparing the coefficients of $z$ in Equation \eqref{cF diff eq}, one can readily determine
\begin{equation}\label{fk eq}
    F_k = c_0q^{k(k-1)/2}.
\end{equation} 
However, there is only one solution with $z$ coefficients given by Equation \eqref{fk eq} (up to scaling by a constant) and it follows that $F(z) = c_0G(z)$. Thus, if $C(qz) = C(z)$, then $C(z) = c_0$, and $C(z)$ has no poles. 
\end{proof}
\begin{corollary}\label{even coror}
Let $C(z)$ be a function defined on $\mathbb{C}\setminus 0$, which is analytic everywhere except for simple poles at $\pm q^k$ for $k \in \mathbb{Z}$. Furthermore, suppose $C(z)$ satisfies $C(qz)=C(z)$. Then, $C(z) = c_1h^0(z) + c_0$, where $c_0$ and $c_1$ are constants and $h^0(z)$ is as defined in Equation \ref{h new def}.
\end{corollary}
\begin{proof}
As both $C(z)$ and $h^0(z)$ have simple poles at $z=-1$ we conclude that there exists a $c_1\neq 0$ such that
\[ \mathrm{Res}(C(-1)) = c_1\mathrm{Res}(h^0(-1)).\]
Furthermore, both $C(z)$ and $h^0(z)$ are invariant under the transformation $z\to qz$, hence for all $k\in\mathbb{Z}$
\[ \mathrm{Res}(C(-q^k)) = c_1\mathrm{Res}(h^0(-q^k)).\]
Thus, the function 
\[D(z) = C(z)-c_1h^0(z),\]
is meromorphic in $\mathbb{C}\setminus 0$, with possible simple poles at $q^k$ for $k \in \mathbb{Z}$, and satisfies $D(qz) = D(z)$. However, by Lemma \ref{lemma 5.6}, $D(z)$ can not have simple poles at $q^k$ for $k \in \mathbb{Z}$. Hence, $D(z)$ is analytic in $\mathbb{C}\setminus 0$ and it follows that $D(z)$ can be written as a convergent Laurent series. Thus,
\[ D(z) = \sum_{j=-\infty}^\infty d_jz^{j} .\]
Substituting this into the $q$-difference equation $D(qz) = D(z)$, we conclude $D(z) = d_0$ $(=c_0)$ and Corollary \ref{even coror} follows immediately.
\end{proof}

\section*{Acknowledgment}
The authors would like to thank Dr. Pieter Roffelsen for helpful discussions during the inception of the paper.

\section*{Funding}
Nalini Joshi's research was supported by Australian Research Council Discovery Projects \#DP200100210 and \#DP210100129. Tomas Lasic Latimer's research was supported the Australian Government Research Training Program and by the University of Sydney Postgraduate Research Supplementary Scholarship in Integrable Systems.

\printbibliography

\end{document}